\documentclass[11pt,a4paper, oneside]{amsart} 

\usepackage{latexsym}

\usepackage[a4paper , lmargin = {2.5cm} , rmargin = {2.5cm} , tmargin = {2.5cm} , bmargin = {2.5cm} ]{geometry} 

\usepackage{geometry}
\usepackage{pinlabel}         
\usepackage{graphicx}
\usepackage{amssymb, amsthm}
\usepackage{epstopdf}
\usepackage{romannum}
\usepackage{tikz}
\usepackage{subfig}
\usepackage{MnSymbol}
\usepackage[arrow, matrix, curve]{xy}

  	\newcommand{\Z}{\ensuremath{\mathbb{Z}}}   
  	\newcommand{\Q}{\ensuremath{\mathbb{Q}}}   
  	\newcommand{\R}{\ensuremath{\mathbb{R}}} 
	\newcommand{\C}{\ensuremath{\mathbb{C}}}

   	\def\SAut{{\rm{SAut}}$(F_n)$}
        
	\def\GL{{\rm{GL}}}
	\def\SL{{\rm{SL}}}

	\def\Z{{\mathbb{Z}}}

	\def\CAT{{\rm{CAT}}$(0)$}

\theoremstyle{plain}
	
\newtheorem*{NewTheoremA}{Theorem A}
\newtheorem*{NewTheoremC}{Theorem C}
\newtheorem*{NewTheoremD}{Theorem D}

\newtheorem*{NewCorollaryB}{Corollary B}

\newtheorem*{NewPropositionE}{Proposition E}

\newtheorem{theorem}{Theorem}[section]
\newtheorem{lemma}[theorem]{Lemma}

\newtheorem{proposition}[theorem]{Proposition}
\newtheorem{definition}[theorem]{Definition}

\title{Representations of groups with CAT(0) fixed point property}
\author{Olga Varghese}
\thanks{Research partially supported by SFB 878.}
\address{Olga Varghese\\
Department of Mathematics\\
M\"unster University\\ 
Einsteinstra\ss e 62\\
48149 M\"unster (Germany)}
\email{olga.varghese@uni-muenster.de}

\begin{document}
\pagenumbering{arabic}
\begin{abstract}
We show that certain representations over fields with positive characteristic of groups having \CAT\ fixed point property ${\rm F}\mathcal{B}_{\widetilde{A}_n}$ have finite image. In particular, we obtain rigidity results for representations of the following groups: the special linear group over $\Z$, $\SL_k(\Z)$, the special automorphism group of a free group, ${\rm SAut}(F_k)$, the mapping class group of a closed orientable surface, ${\rm Mod}(\Sigma_g)$, and many other groups.

In the case of characteristic zero we show that low dimensional complex representations of groups having \CAT\ fixed point property ${\rm F}\mathcal{B}_{\widetilde{A}_n}$ have finite image if they always have compact closure.
\end{abstract}
\maketitle

\section{Introduction and statements}
In this paper we study linear representations of groups with certain \CAT\ fixed point property. Let $\mathcal{B}_{\widetilde{A}_n}$ be the class of buildings of type $\widetilde{A}_n$. 
Roughly speaking, a building of type $\widetilde{A}_n$ is a highly symmetrical $n$-dimensional \CAT\ simplicial complex with a vertex coloring. 
We say that a group $G$ has property ${\rm F}\mathcal{B}_{\widetilde{A}_n}$ if any  simplicial color-preserving action of $G$ on any member of $\mathcal{B}_{\widetilde{A}_n}$ has a fixed point.  Our investigations on property ${\rm F}\mathcal{B}_{\widetilde{A}_n}$ are motivated by Serre's property ${\rm F}\mathcal{A}$. Recall that a group $G$ is said to satisfy Serre's property ${\rm F}\mathcal{A}$ if every action, without inversions, of $G$ on a simplicial tree has a fixed point. 
 Note that if a group $G$ has Serre's property ${\rm F}\mathcal{A}$ it  has also property 
${\rm F}\mathcal{B}_{\widetilde{A}_1}$, since buildings of type $\widetilde{A}_1$ are trees without leaves. 

We show that property ${\rm F}\mathcal{B}_{\widetilde{A}_n}$ strongly affects  the representation theory of groups. The following theorem illustrates this fact.

\begin{NewTheoremA}
Let $G$ be a finitely generated group with property ${\rm F}\mathcal{B}_{\widetilde{A}_n}$ and $\rho:G\rightarrow{\rm GL}_{n+1}(K)$ a linear representation over a field $K$ with positive characteristic. Then $\rho(G)$ is finite.
\end{NewTheoremA}

As an application of Theorem A we obtain the following results.
\begin{NewCorollaryB}
Let $K$ be a field with positive characteristic. Then every linear representation
\begin{enumerate}
\item[(i)] $\SL_k(\Z)\rightarrow\GL_n(K)$ for $k\geq 3$ and all $n\in\mathbb{N}$,
\item[(ii)] $\SL_k(\Z[\frac{1}{p}])\rightarrow\GL_n(K)$ for $p$ prime number, $k\geq 3$ and $n\leq k-1$,
\item[(iii)] ${\rm SAut}(F_k)\rightarrow\GL_n(K)$ for $k\geq 3$ and $n\leq k-1$,
\item[(iv)] ${\rm Mod}(\Sigma_g)\rightarrow \GL_n(K)$ for $g\geq 2$ and $n\leq g$,
\item[(v)] $W\rightarrow\GL_d(K)$ for $d\leq n$ and $W$ is a f.g. Coxeter group such that every special parabolic subgroup of rank $n$ is finite.
\end{enumerate}
has finite image.
\end{NewCorollaryB}
Results of similar flavor were proved in
\cite{Varghese1} 
where it is shown that low dimensional representations of ${\rm SAut}(F_k)$ over fields of characteristic not equal $2$ are trivial. Concerning  the mapping class group of a closed orientable surface of genus $g$, ${\rm Mod}(\Sigma_g)$, it is  proved in \cite{Button} that this group has no faithful linear representation in any dimension over any field of positive characteristic.

Let us further mention that Theorem A is not true for fields of characteristic zero, as we can 
consider the group ${\rm SL}_n(\Z)$ which has property ${\rm F}\mathcal{B}_{\widetilde{A}_{n-1}}$  but the canonical embedding $\SL_n(\Z)\hookrightarrow\GL_n(\R)$ has infinite image.

In the case of characteristic zero we prove:
\begin{NewTheoremC}
Let $G$ be a finitely generated group with property  ${\rm F}\mathcal{B}_{\widetilde{A}_n}$ and $\rho:G\rightarrow{\rm GL}_{n+1}(\mathbb{Q})$ a linear representation. If there exists a positive definite quadratic form on $\Q^{n+1}$ which is invariant under $\rho(G)$, then $\rho(G)$ is finite.
\end{NewTheoremC}

Further, we generalize a theorem of Alperin \cite{Alperin}, who proved the following result for $n=1$. 
\begin{NewTheoremD}
Let $G$ be a finitely generated group with property  ${\rm F}\mathcal{B}_{\widetilde{A}_n}$. If  every $(n+1)$-dimensional representation $\rho:G\rightarrow{\rm GL}_{n+1}(\C)$ has image with compact closure, then the image $\rho(G)$ is finite for every $(n+1)$-dimensional representation.
\end{NewTheoremD}

The general strategy of the proofs consists of three steps.
First, as Farb noted in \cite[1.7]{Farb} one can generalize Serre's result \cite[6.2.22]{Serre} as follows.

\begin{NewPropositionE}
\label{Coefficients}
Let $G$ be a finitely generated group with property ${\rm F}\mathcal{B}_{\widetilde{A}_n}$ and let $\rho:G\rightarrow\GL_{n+1}(K)$ be a linear representation over a field $K$. 
\begin{enumerate}
\item[(i)] If $K$ has positive characteristic, then the eigenvalues of $\rho(g)$ for $g\in G$ are roots of unity.
\item[(ii)] If $K$ has characteristic zero, then the coefficients of the characteristic polynomial of $\rho(g)$ for $g\in G$ and his roots  are integral over $\mathbb{Z}$.
\end{enumerate}
\end{NewPropositionE}
We include a proof of this proposition here which uses the same ideas as the proof in \cite[6.2.22]{Serre}. Then we show that each element in the image of a linear representation $\rho$ has finite order. Hence, the image of $\rho$ is a finitely generated torsion linear group which is by Schur's Theorem finite (\cite{Schur}, \cite[9.9]{Lam}). 

\section{General results of property ${\rm F}\mathcal{B}_{\widetilde{A}_n}$}
Definitions and properties concerning Coxeter groups and buildings of type 
 $\widetilde{A}_n$  can be found in \cite{Abramenko}, \cite{Humphreys}, \cite{Ronan}.

We need the following crucial definition. 
\begin{definition}
A group $G$ has  property ${\rm F}\mathcal{B}_{\widetilde{A}_n}$ if any simplicial type preserving action on any building of type $\widetilde{A}_n$ has a fixed point. 
\end{definition}

An important fact about buildings of type $\widetilde{A}_n$ is that its geometric realisation has a natural metric that satisfies the \CAT\ curvature condition, for more details see (\cite[11.16]{Abramenko}) and \cite{Haefliger}.

Let us compare property ${\rm F}\mathcal{B}_{\widetilde{A}_n}$ to the fixed point property ${\rm F}\mathcal{A}_n$ which was defined by Farb in \cite{Farb}. 
\begin{definition}
A group $G$ has property $F\mathcal{A}_n$ if any simplicial action of $G$ on any $n$-dimensional  simplicial \CAT\ complex has a fixed point (in the geometric realisation). 
\end{definition}
Since the geometric realisation of a building of type $\widetilde{A}_n$ is a $n$-dimensional \CAT\ simplicial complex, we immediately obtain the following result.
\begin{lemma}
\label{FAFB}
If a group $G$ has property ${\rm F}\mathcal{A}_n$, then it has property ${\rm F}\mathcal{B}_{\widetilde{A}_n}$.
\end{lemma}

There are many  groups known to have property ${\rm F}\mathcal{A}_n$, hence by Lemma \ref{FAFB} these groups also have property ${\rm F}\mathcal{B}_{\widetilde{A}_n}$.

\begin{proposition}
\label{Examples}
\begin{enumerate}
\item[(i)] For $n\geq 3$ and all $k\in\mathbb{N}$ the group $\SL_n(\Z)$ has property ${\rm F}\mathcal{A}_k$,  (\cite[p. 3]{Farb}).
\item[(ii)] For $n\geq 3$ and $p$ prime number the group $\SL_n(\Z[\frac{1}{p}])$ has property ${\rm F}\mathcal{A}_{n-2}$, (\cite[1.2]{Farb}).
\item[(iii)] For $n\geq 3$ the group ${\rm SAut}(F_n)$ has property ${\rm F}\mathcal{A}_{n-2}$, (private communication with Bridson).
\item[(iv)] For $g\geq 2$ the group ${\rm Mod}(\Sigma_g)$ has property ${\rm F}\mathcal{A}_{g-1}$, (\cite[Theorem A]{Bridson}).
\item[(v)] Let $W$ be a Coxeter group, such that each special parabolic subgroup of rank $n$ is finite, then $W$ has property ${\rm F}\mathcal{A}_{n-1}$ (\cite[1.1]{Barnhill}).
\end{enumerate}
\end{proposition}

The above proposition shows that Corollary B follows by Theorem A.

For a  proof of Proposition E we will need the following result.
\begin{lemma}
\label{noZ}
Let $G$ be a group with property ${\rm F}\mathcal{B}_{\widetilde{A}_n}$. Then $G$ has no quotient isomorphic to $\Z$.
\end{lemma}
\begin{proof}
Let us assume that $G$ has a quotient isomorphic to $\Z$. Then there exists a surjective homomorphism $\Phi: G\twoheadrightarrow \Z$.
Let $(W,I)$ be the Coxeter system of type $\widetilde{A}_n$ and $\Sigma(W,I)$ the associated Coxeter complex (thin building of type $\widetilde{A}_n$). Since $W$ is an affine Coxeter group, there exists an element $w_0\in W$ with infinite order. 
The action
\[ 
\Psi\circ\Phi: G\twoheadrightarrow \Z\cong \langle w_0 \rangle \rightarrow Aut(\Sigma(W,I))
\]
where $\Psi$ acts via left multiplication is type preserving. Let $wW_J$ for $J\subset I$ be an arbitrary simplex in $\Sigma(W,I)$. Then ${\rm stab}(wW_J)=wW_Jw^{-1}$. Since $J\neq I$, the special parabolic subgroup $W_J$ is finite and therefore $wW_Jw^{-1}$ is finite. Thus the action $\Psi\circ\Phi$  has no fixed point which is a contradiction.
\end{proof}

We now state an useful elementary fact.
\begin{proposition}
If a finitely generated group $G$ has property ${\rm F}\mathcal{B}_{\tilde{A}_n}$, then the abelianization $G^{ab}=G/[G,G]$ of $G$ is finite.
\end{proposition}
\begin{proof}
Let us consider the following natural map:
\[
\pi:G\twoheadrightarrow G^{ab}
\]
Hence $G$ is a finitely generated group, its abelianization is a finitely generated abelian group. Therefore by the fundamental theorem of finitely generated abelian groups we have the following decomposition
\[
G^{ab}\cong \Z^{k}\times Tor(G),
\]
where $Tor(G)$ is a finite group. Assume that $k\geq 1$, then we get a surjective map
\[
G\twoheadrightarrow G^{ab}\cong \Z^{k}\times Tor(G)\twoheadrightarrow \mathbb{Z}
\]
which is by Lemma \ref{noZ} a contradiction.
\end{proof}

\section{Proofs}
\subsection{Proof of Proposition E}
Before we can prove Proposition E we need a result by Grothendieck. 
Recall that a field extension $L/K$ is said to be finitely generated over $K$ if there exist finitely many elements $a_1,\ldots, a_l$ in $L$ such that the smallest field extension $K(a_1, \ldots, a_l)$ containing $a_1, \ldots, a_l$ equals $L$. 

For finitely  generated  field extensions $L/\mathbb{F}_p$, where we denote by $\mathbb{F}_p$ the prime field of $L$, there is a good description on integral (algebraic) closure of $\mathbb{F}_p$ in $L$, and similarly for fields extensions of $\mathbb{Q}$. 

\begin{proposition}(\cite[p.140 7.1.8]{Grothendieck})
\label{Grothendieck}
\begin{enumerate}
\item[(i)] Let $L$ be an infinite field of characteristic $p$. If $L$ is finitely generated over $\mathbb{F}_p$, then the integral closure of $\mathbb{F}_p$ in $L$ is equal to the intersection of all discrete valuation rings in $L$.
\item[(ii)] Let $L$ be a field of characteristic $0$. If $L$ is finitely generated over $\mathbb{Q}$, then the integral closure of $\mathbb{Z}$ in $L$ is equal to the intersection of all discrete valuation rings in $L$.
\end{enumerate}
\end{proposition}

Now we are ready to prove Proposition E.
\begin{proof}
Let $\rho:G\rightarrow\GL_{n+1}(K)$
be a linear representation over a field $K$. Let $\left\{g_1, \ldots, g_l\right\}$ be a finite generating set of $G$ and 
 let $K_{\rho}$ be the subfield of $K$ generated by the coefficients of the matrices $\rho(g_i)$ for $i=1,\ldots, l$. We get $\rho(G)\subseteq{\rm GL}_{n+1}(K_\rho)$. Further, the field $K_\rho$ is finitely generated over its prime field $\mathbb{F}$. Therefore there exist elements $a_1, \ldots, a_m\in K_\rho$ such that $K_\rho=\mathbb{F}(a_1,\ldots, a_m)$. 
 
If the characteristic of $K$ is positive and the elements $a_1, \ldots, a_m$ are algebraic over $\mathbb{F}$, then $K_\rho$ is a finite field and hence $\rho(G)$ is finite. 

Otherwise  $K_\rho$ is infinite. Since $K_\rho$ is finitely generated  over $\mathbb{F}$ and infinite, there exists a discrete valuation $\nu$ on $K_\rho$. Let $\mathcal{O}_{\nu}$ be the corresponding valuation ring and $\Delta(K_\rho^{n+1}, \nu)$ be the associated building of type $\widetilde{A}_n$, for details see \cite[chap. 9]{Ronan}. Let $\GL_{n+1}(K_\rho)^\circ$ be the kernel of the homomorphism
\[
\nu\circ\det:\GL_{n+1}(K_\rho)\rightarrow \Z.
\]
Since $\rho(G)$ has property $F\mathcal{B}_{\widetilde{A}_n}$, by Lemma \ref{noZ} this group  has no quotient isomorphic to $\Z$. This shows that $\rho(G)$ is contained in $\GL_{n+1}(K_\rho)^\circ$, which acts by type preserving automorphisms on $\Delta(K_\rho^{n+1}, \nu)$.  Since $\rho(G)$ has property ${\rm F}\mathcal{B}_{\widetilde{A}_n}$, there is a vertex $x$ of $\Delta(K_\rho^{n+1}, \nu)$ which is  invariant under $\rho(G)$. Hence $\rho(G)$ is contained in the stabilizer of $x$. All vertex stabilizer are conjugate to $\GL_{n+1}(\mathcal{O}_\nu)$.
Thus for each $g\in G$ the coefficients of the characteristic polynomial of $\rho(g$) belong to $\mathcal{O}_\nu$.  We can do this construction for each discrete valuation on $K_\rho$, hence the coefficients of the characteristic polynomial of $\rho(g)$  lie in the intersection $\bigcap\limits_{\nu\text{ disc. val. on }K_\rho} \mathcal{O}_\nu$.

If the characteristic of $K$ is zero, then by Proposition \ref{Grothendieck}(ii) this intersection is equal to the integral closure of $\Z$ in $K_{\rho}$ and hence the coefficients of the characteristic polynomial and his roots  are integral over $\Z$.

If the characteristic of $K$ is $p$,  then by Proposition \ref{Grothendieck}(i) the intersection $\bigcap\limits_{\nu\text{ disc. val. on }K_\rho} \mathcal{O}_\nu$ is equal to the integral closure of $\mathbb{F}_p$ in $K_{\rho}$ and hence the coefficients of the characteristic polynomial and his roots are integral (algebraic) over $\mathbb{F}_p$. Let $\lambda$ be an eigenvalue of $\rho(g)$. Then there exists  a monic polynomial $f$ in $\mathbb{F}_p[X]$  with $f(\lambda)=0$. Thus the field extension $[\mathbb{F}_p(\lambda):\mathbb{F}_p]\leq {\rm deg}(f)$ is finite, therefore $\mathbb{F}_p(\lambda)$ is a finite field of order $q=p^k$ for suitable $k\in\mathbb{N}$. Since $\lambda\in\mathbb{F}_p(\lambda)^*$,  it follows by Lagrange's Theorem that $\lambda^{q-1}=1$. Hence $\lambda$ is a root of unity.
\end{proof}

Before we can prove the main theorems we need Schur's result about linear torsion groups. 
\begin{theorem}(\cite[9.9]{Lam})
\label{Schur}
Let $K$ be an arbitrary field and $n\in\mathbb{N}$. Then every finitely generated torsion subgroup of ${\rm GL}_n(K)$ is finite. 
\end{theorem}

\subsection{Proof of Theorem A}

\begin{proof}
Let $\rho:G\rightarrow\GL_{n+1}(K)$
be a linear representation over a field $K$ with characteristic $p$.

Let $\left\{g_1, \ldots, g_l\right\}$ be a finite generating set of $G$ and 
 let $K_{\rho}$ be the subfield of $K$ generated by the coefficients of the matrices $\rho(g_i)$ for $i=1,\ldots, l$. We get $\rho(G)\subseteq{\rm GL}_{n+1}(K_\rho)$. 
Our first goal is to show that the order of $\rho(g)$ for $g\in G$ is finite.

Let $K^{alg}_\rho$ be the algebraic closure of $K_\rho$ and 
\[
\iota: \GL_{n+1}(K_\rho)\rightarrow \GL_{n+1}(K^{alg}_\rho)
\]
be the canonical embedding.  
Since $K^{alg}_\rho$ is algebraically closed field, the matrix $\iota\circ\rho(g)$ is triangularizable. Therefore there exist $A, B \in\GL_{n+1}(K^{alg}_\rhoƒ)$ with
\[
\iota\circ\rho(g)=A\cdot B \cdot A^{-1}
\]
where $B$ is of the form
 \[
   B=
  \left[ {\begin{array}{cccc}
   \lambda_1 & * & *&\ldots\\
   0 & \lambda_2 &*&\ldots\\
   0 & 0 &\lambda _3&\ldots\\
   \ldots &\ldots&\ldots&\ldots\\
  \end{array} } \right]
\]
where $\lambda_1, \ldots, \lambda_{n+1}$ are the eigenvalues of $\iota\circ\rho(g)$. The eigenvalues of $\iota\circ\rho(g)$ are by Proposition E(i) roots of unity, therefore the  order of $\lambda_j$ is finite. Let $k=lcm\left\{ o(\lambda_j)\mid j=1, \ldots, n+1\right\}$. Then 
\[
  B^k=
  \left[ {\begin{array}{cccc}
   1 & * & *&\ldots\\
   0 & 1 &*&\ldots\\
   0 & 0 & 1&\ldots\\
   \ldots &\ldots&\ldots&\ldots\\
  \end{array} } \right]
=\left[ {\begin{array}{cccc}
   1 & 0 & 0&\ldots\\
   0 & 1 &0&\ldots\\
   0 & 0 & 1&\ldots\\
   \ldots &\ldots&\ldots&\ldots\\
  \end{array} } \right]
+
\left[ {\begin{array}{cccc}
   0 & * & *&\ldots\\
   0 & 0 &*&\ldots\\
   0 & 0 & 0&\ldots\\
   \ldots &\ldots&\ldots&\ldots\\
  \end{array} } \right]
\]
and by Frobenius homomorphism we get 
\[
  B^{k\cdot p^l}=
  \left[ {\begin{array}{cccc}
   1 & 0 & 0 &\dots\\
   0 & 1 & 0 &\ldots\\
   0 & 0 & 1 &\ldots\\
   \ldots &\ldots&\ldots&\ldots\\
  \end{array} } \right]
={\rm Id}
\]
for $l\in\mathbb{N}$ with $p^l\geq n+1$. Therefore the order of $B$ and hence of $\iota\circ\rho(g)$ and $\rho(g)$ is finite.

So $\rho(G)$ is a finitely generated torsion group and by Theorem \ref{Schur} this group is finite.
\end{proof}

\subsection{Proof of Theorem C}

\begin{proof}
Let $\rho:G\rightarrow{\rm GL}_{n+1}(\mathbb{Q})$ be a linear representation of $G$. Since there exists a positive definite quadratic form on $\Q^{n+1}$ which is invariant under $\rho(G)$, the image of $\rho$ is contained in a conjugate of orthogonal matrices and therefore the eigenvalues of $\rho(g)$ for $g\in G$  have absolute value $1$.

By Proposition E(ii) the coefficients of the characteristic polynomial  of $\rho(g)$, denoted by $\chi_{\rho(g)}$, are integral over $\Z$ in $\Q$, hence $\chi_{\rho(g)}\in\Z[X]$. 
The set of all monic polynomials of degree $n+1$ with integer coefficients having roots with absolute value one is finite. To see this, we write
\[
X^{n+1}+a_{n}X^n+\ldots+a_0=\prod_{j=1}^{n+1}(X-\lambda_j)
\]
where $a_j\in\Z$ and $\lambda_j$ are the roots of the polynomial.

Since the $a_j$ are integers, each $a_j$ is limited to at most $2\cdot\binom{n}{j}+1$ values and therefore the number of polynomial is finite and hence the set 
\[
E:=\left\{ \lambda\in\C\mid \lambda\text{ is an eigenvalue of }\rho(g)\text{ for }g\in G\right\}
\]
is finite. Thus each $\lambda\in E$ has finite order.
Let $g\in G$, then $\iota\circ\rho(g)$ is diagonalizable, where 
$\iota:{\rm GL}_{n+1}(\mathbb{Q})\rightarrow{\rm GL}_{n+1}(\mathbb{C})$ is the canonical embedding. Since each eigenvalue has finite order the matrices $\iota\circ\rho(g)$ and $\rho(g)$ have finite order. So $\rho(G)$ is a finitely generated torsion group and by Theorem \ref{Schur} this group is finite. 
\end{proof}

\subsection{Proof of Theorem D}
\begin{proof}
Let $\rho:G\rightarrow{\rm GL}_{n+1}(\C)$ be a linear representation of $G$.
Let $\left\{g_1, \ldots, g_l\right\}$ be a finite generating set of $G$ and 
 let $\C_{\rho}$ be the subfield of $\C$ generated by the coefficients of the matrices $\rho(g_i)$ for $i=1,\ldots, l$. We get $\rho(G)\subseteq{\rm GL}_{n+1}(\C_\rho)$. 
Our goal is to show, that the eigenvalues of $\rho(g)$ for $g\in G$ are roots of unity.

It follows by Proposition E(ii) that the eigenvalues of any element of $\rho(G)$ are in some number field. Since $\rho$ has image with compact closure, the eigenvalues of $\rho(g)$ for $g\in G$ have absolute value $1$ and since each representation of degree $n+1$ has image with compact closure it follows by Dirichlet Unit Theorem \cite[6.1.6]{Ash} that the eigenvalues of $\rho(g)$ for $g\in G$ are roots of unity. 

Further $\rho(g)$ is diagonalizable and since each eigenvalue has finite order the matrix $\rho(g)$ has finite order. So $\rho(G)$ is a finitely generated torsion group and by Theorem \ref{Schur} this group is finite. 
\end{proof}

\subsection*{Acknowledgements} The author would like to thank the referee for many helpful comments.

\end{document}